\documentclass[a4paper,11pt]{amsart}

\usepackage{epsfig}
\usepackage{amsthm,amsfonts}
\usepackage{amssymb,graphicx,color}
\usepackage[all]{xy}
\usepackage{verbatim}
\usepackage{hyperref}
\usepackage{enumitem}

\newtheorem{theorem}{Theorem}[section]
\newtheorem*{theorem*}{Theorem}
\newtheorem{lemma}[theorem]{Lemma}
\newtheorem{corollary}[theorem]{Corollary}
\newtheorem{proposition}[theorem]{Proposition}
\newtheorem{definition}[theorem]{Definition}

\newtheorem{claim}{Claim}[theorem]
\newtheorem{remark}[theorem]{Remark}
\newtheorem{problem}[theorem]{Problem}

\newtheorem{innercustomthm}{Theorem}
\newenvironment{customthm}[1]
  {\renewcommand\theinnercustomthm{#1}\innercustomthm}
  {\endinnercustomthm}

\newcommand{\C}{\mathbb{C}}

\begin{document}

\title[Proof of the metric Arnold's corank problem]
{Proof of the metric Arnold's corank problem}

\author[A. Fernandes]{Alexandre Fernandes}
\author[Z. Jelonek]{Zbigniew Jelonek}
\author[J. E. Sampaio]{Jos\'e Edson Sampaio}

\address[Zbigniew Jelonek]{ Instytut Matematyczny, Polska Akademia Nauk, \'Sniadeckich 8, 00-656 Warszawa, Poland  
             \& ICMC-USP, Caixa Postal 668, 13560-970, Sao-Carlos, Brazil \newline
              E-mail: {\tt najelone@cyf-kr.edu.pl}
}

{\address[Jos\'e Edson Sampaio \& Alexandre Fernandes]{    
              Departamento de Matem\'atica, Universidade Federal do Cear\'a,
	      Rua Campus do Pici, s/n, Bloco 914, Pici, 60440-900, 
	     Fortaleza-CE, Brazil. \newline  
               E-mail: {\tt edsonsampaio@mat.ufc.br, alex@mat.ufc.br}
}

\thanks{The first named author was partially supported by CNPq-Brazil grant 304700/2021-5. 
The second named author is partially supported by the grant of Narodowe Centrum Nauki, number 2024/55/B/ST1/01412.
The third named author was partially supported by CNPq-Brazil grant 303375/2025-6 and by the Serrapilheira Institute (grant number Serra -- R-2110-39576).
}

\begin{abstract}
In this article, we approach the Arnold corank problem, posed by Arnold in 1975, which asks whether the corank of holomorphic functions is an ambient topological invariant. Here, we obtain a complete positive answer to the metric Arnold corank problem, which asks whether the corank of holomorphic functions is an ambient bi-Lipschitz invariant. Consequently, we show that for complex hypersurfaces, the multiplicity equal to two is an ambient bi-Lipschitz invariant. We also prove that the Arnold corank problem holds true for holomorphic functions of three variables. Other topological invariants are also presented. 
\end{abstract}

\maketitle
\tableofcontents

\section{Introduction}

Let $f\colon (\C^n,0)\to (\C,0)$ be the germ of a holomorphic function at the origin. Let $(V(f),0)$ be the germ of the zero set of $f$ at the origin. We recall that the {\bf corank} of the germ of a holomorphic function $f\colon (\C^n,0)\to (\C,0)$ is the corank of the Hessian matrix $f$ at $0$. We denote the corank of $f$ by ${\rm crk}(f)$. When $f$ is a reduced function, the {\bf corank} of $(V(f),0)$ is ${\rm crk}(V(f),0):={\rm crk}(f)$, and the {\bf rank} of $(V(f),0)$ is ${\rm rk}(V(f),0):=n-{\rm crk}(f)$.

In 1975, Arnold asked whether the corank of a holomorphic function is an ambient topological invariant. More precisely, we have the following problem:
 
\begin{problem}[Arnold corank problem] 
Let $f,g\colon (\C^n,0)\to (\C,0)$ be germs of reduced holomorphic functions. If there is a homeomorphism $\varphi\colon(\C^n,0)\to (\C^n,0)$ such that $\varphi(V(f))=V(g)$, is it true that ${\rm crk}(f)={\rm crk}(g)$?
\end{problem}
This problem is also stated in the book Arnold's Problems (see \cite[Problem 1975-14, p. 22]{Arnold:2004}). In this book, there is also a comment from Vassiliev that nothing is known to him about this problem (see \cite[p. 284]{Arnold:2004}).

So, it is natural to ask the following question.

\begin{problem}[Metric Arnold corank problem]\label{prob:Lip_rank}
Let $f,g\colon (\C^n,0)\to (\C,0)$ be germs of reduced holomorphic functions. If there is a bi-Lipschitz homeomorphism $\varphi\colon(\C^n,0)\to (\C^n,0)$ such that $\varphi(V(f))=V(g)$, is it true that ${\rm crk}(f)={\rm crk}(g)$?
\end{problem}

The investigation of metric versions of the Zariski multiplicity problem (which we will discuss later) and the Arnold corank problem is a fully plausible and important line of inquiry. When these classical problems were first formulated, it had not yet been established that the topological study and the Lipschitz study of complex hypersurfaces were, in fact, distinct. Now that this distinction is known, it is entirely natural to study their metric versions.

The corank of a holomorphic function at a singular point measures how far a singularity is from being a Morse singularity (which corresponds to corank zero). This measure is operationally defined by the powerful Splitting Lemma (also known as the Generalized Morse Lemma), which shows that any holomorphic function near a critical point can be decomposed into a quadratic part plus a residual function depending only on the corank variables, i.e., given a holomorphic function $f\colon (\C^n,0)\to(\C,0)$ with $r=n-{\rm crk}(f)$, there is a holomorphic diffeomorphism $\psi\colon (\C^n,0)\to (\C^n,0)$ such that
$$
f\circ \psi(z)=z_1^2+\cdots+z_r^2+h(z_{r+1},...,z_n),
$$
where $h\colon (\C^{n-r},0)\to (\C,0)$ is a holomorphic function such that ${\rm ord}_0(h)>2$.
In this way, the corank determines the dimension of the “essential” singularity space that carries the true complexity of the singularity. 
Therefore, the Arnold corank problem (resp. the metric Arnold corank problem) asks whether the decomposition given by the Splitting Lemma, which is, a priori, an analytic invariant, depends only on the embedded topological (resp. bi-Lipschitz) type of the zero sets of the involved functions.

In this paper, we provide a complete positive answer to the metric Arnold corank problem.

\begin{customthm}{\ref*{thm:Lip_rank}}
Let $f,g\colon (\C^n,0)\to (\C,0)$ be two reduced holomorphic functions which are singular at the origin. Assume that there is a bi-Lipschitz homeomorphism $\phi\colon (\C^n,0)\to (\C^n,0)$ that sends the zeros of $f$ onto the zeros of $g$. Then ${\rm crk}(f)={\rm crk}(g)$.
\end{customthm}

We also provide a positive answer to the Arnold corank problem when $n=3$.

\begin{customthm}{\ref*{thm:top_rank}}
Let $f,g\colon (\C^3,0)\to (\C,0)$ be two reduced holomorphic functions. Assume that there is a homeomorphism $\phi\colon (\C^3,0)\to (\C^3,0)$ sending the zeros of $f$ onto the zeros of $g$. Then ${\rm crk}(f)={\rm crk}(g)$.
\end{customthm}

Another problem addressed here is the following conjecture, which is a metric version of the famous Zariski multiplicity problem (see \cite{zar}):

\vspace{3mm}

\begin{enumerate}[leftmargin=0pt]
\item[]{\bf Conjecture A.} 
{\it Let $X,Y\subset \C^n$  be two complex analytic hypersurfaces. If their germs at zero are ambient bi-Lipschitz homeomorphic, then their multiplicities $m(X,0)$ and $m(Y,0)$ are equal.}
\end{enumerate}

\vspace{3mm}

Despite  many efforts (see e.g. \cite{BobadillaFS:2018}, \cite{comte}, \cite{comteetall},  \cite{eph}, \cite{gau}, \cite{risler}), this conjecture is still open. In \cite{FernandesJS:2023}, we have proved that this conjecture follows from the following metric version of Zariski problem B (see \cite{zar}):
\vspace{3mm}

\begin{enumerate}[leftmargin=0pt]
\item[]{\bf Conjecture B.} {\it Let $C(X),C(Y)\subset \C^n$ be two $(n-1)-$dimensional algebraic cones with bases $X,Y.$  If there is a bi-Lipschitz homeomorphism 
$$\varphi\colon(\mathbb{C}^n,C(X),0)\to (\mathbb{C}^n,C(Y),0),$$ 
then $X$ is homeomorphic to $Y.$}
\end{enumerate}

\vspace{3mm}
In this paper, we  answer the above question partially: we show that if two algebraic $(n-1)-$dimensional cones $C(X), C(Y) \subset\mathbb C^n$ with isolated singularities are homeomorphic, then they have the same degree (see Proposition \ref{prop:1}) and, in particular, $X$ is homeomorphic to $Y.$
We prove that if two  algebraic $(n-1)$-dimensional cones $C(X), C(Y)\subset\mathbb C^n$  are ambient homeomorphic, then their bases $X$ and $Y$ have the same Euler characteristic (see Proposition \ref{prop:euler}).

As a consequence of Theorem \ref{thm:Lip_rank}, we obtain that Conjecture A is true for multiplicity two, that is, if $(X,0),(Y,0)\subset (\C^n,0)$ are germs of analytic hypersurfaces that are ambient bi-Lipschitz equivalent and $m(X,0)=2$, then also $m(Y,0)=2$ (see Corollary \ref{cor:nowe}).
In particular, this implies that Conjecture B is true if the degree of $C(X)$ is $2.$

\section{Algebraic Cones}
We start with a definition:

\begin{definition}
Let $X\subset \mathbb {CP}^n$ be an algebraic variety. We view $\mathbb {CP}^n$  as a hyperplane at infinity of $\mathbb {CP}^{n+1}$.  Then the \emph{ algebraic cone} $\overline{C(X)}\subset \mathbb {CP}^{n+1}$ with base $X$ is the set
$$\overline{C(X)}=\bigcup_{x\in X} \overline{O,x},$$ 
where $O$ is the center of coordinates in $\mathbb C^{n+1}\subset \mathbb {CP}^{n+1}$, and $\overline{O,x}$ is the projective line  through $O$ and $x.$ The  \emph{affine cone} $C(X)$ is $\overline{C(X)}\setminus X.$
The link of $C(X)$ is the set $L=\{ x\in C(X): ||x||=1\}.$
\end{definition}

Given a holomorphic function $f\colon (\C^n,0)\to \C$, we denote $V(f)=f^{-1}(0)$ and $F_f$ as the Milnor fibre of $f$. 
We recall that \emph{the multiplicity of} $V(f)$ at the origin, denoted by $m(V(f),0)$, is defined as follows: we write
$$f=f_m+f_{m+1}+\cdots+f_k+\cdots,$$
where each $f_k$ is the zero polynomial or a homogeneous polynomial of degree $k$, and $f_m\neq 0$. Then, $m(V(f),0):= m$. We also denote $in(f):=f_m$.

The following result is well-known (see, e.g., \cite{max}).

\begin{theorem}\label{hyp}
Let $V\subset \C\mathbb P^{n+1}$ be a smooth algebraic hypersurface of degree $d$. Then the integral (co)homology of V is torsion free, and the corresponding Betti numbers are  as follows:

(1) $b_i(V)=0$ if $i\not=n  \  is \ odd \ or \ i\not\in [0,2n].$

(2) $b_i(V)=1$ if $i\not=n$\ is even \ and $i\in [0,2n]. $

(3) $b_n(V)=\frac{(d-1)^{n+2}+(-1)^{n+1}}{d} + \frac{3(-1)^{n}+1}{2}.$
\end{theorem}

Now we are ready to prove:

\begin{proposition}\label{prop:1}
Let $P=C(X),R=C(Y) \subset \C^{n+1}$ be two algebraic $n-$dimensional cones with smooth bases $X,Y.$ If $P, R$ are homeomorphic, then  $\deg  \ P = \deg\ R.$ In particular, $X$ and $Y$ are also homeomorphic.
\end{proposition}

\begin{proof}
Let deg $P=p$ and deg $R=r.$ 

Since the theorem is true for $1-$dimensional cones, we can assume that $P$ and $R$ have connected bases. By \cite{prill}, we can assume that $0$ is not a topologically regular point of $P$ and $R$, because otherwise both cones are hyperplanes. Thus, $p,r\geq 2$. Hence $P^*=P\setminus \{0\}$ is homeomorphic to $R^*=R\setminus\{0\}.$ Let  $L_P, L_R$ be the links of $P$ and $R$, respectively. Hence $L_P$ is a deformation retract of $P^*$, and similarly $L_R$ is a deformation retract of $R^*$.
Since $P^*$ and $R^*$ are homeomorphic we know that $L_P$ is homotopy equivalent to $L_R.$ Let $X,Y$ be bases of $P,R.$

We have the Hopf fibrations $\pi : L_P \to X, \ \pi': L_R\to Y$ whose fibers are circles.
Analizing  the spectral sequences of the mappings $\pi$ and $\pi'$ (see  \cite{orlik}, \cite{kol}), we find that the corresponding Betti numbers are given as follows:

(1) $b_i(L_P, \mathbb Q) = b_i(X,\mathbb Q)-b_{i-2}(X,\mathbb Q)$  if $i\le {\rm dim} \ X$ 

(2)  $b_{i+1}(L_P, \mathbb Q) = b_i(X,\mathbb Q)-b_{i+2}(X,\mathbb Q)$ if $i\ge {\rm dim} X.$

(3) $b_i(L_R, \mathbb Q) = b_i(Y,\mathbb Q)-b_{i-2}(Y,\mathbb Q)$  if $i\le {\rm dim} \ X$ 

(4)  $b_{i+1}(L_R, \mathbb Q) = b_i(Y,\mathbb Q)-b_{i+2}(Y,\mathbb Q)$ if $i\ge {\rm dim} X.$

Since $P,R$  have homotopic links $L_P\sim L_R$, we see by (1) - (4)  and Theorem \ref{hyp}  that
$b_n(X,\mathbb Q)=b_n(Y,\mathbb Q)$, i.e., $$\frac{(p-1)^{n+1}+(-1)^{n}}{p}=\frac{(r-1)^{n+1}+(-1)^{n}}{r}.$$ 

Since the function $f(x)=\frac{(x-1)^{n+1}+(-1)^{n}}{x}$ increases for $x\ge 2$, we have $p=r$.
%

\end{proof}

\begin{remark}
{\rm From our proof, it follows that $\deg P=\deg R$ if  $P^*$ is homotopy equivalent to $R^*.$ We will use this version of Proposition \ref{prop:1}   in the next section}.
\end{remark}

We see that homeomorphic cones of dimension $n$ in $\C^{n+1}$ with smooth bases have homeomorphic bases. If the bases are non-smooth, the problem of whether they are homeomorphic is a difficult open problem that implies the metric version of Zariski multiplicity conjecture (see \cite{FernandesJS:2023}). However,  in the general case, we can prove  the following result, which slightly supports a positive answer to this problem:

\begin{proposition}\label{prop:euler}
Let $P=C(X), Q=C(Y)\subset \C^{n+1}$ be $n-$dimensional algebraic cones, which are ambient homeomorphic. Then $\chi(X)=\chi(Y).$
\end{proposition}

\begin{proof}
Let $f=0,\ g=0$ be the reduced equations of the cones $P$ and $Q$, respectively, and let $d_f=m(V(f),0), \ d_g=m(V(g),0).$ The global  fibrations given by $f,\ g$ are equivalent to local ones. Hence, to compute the Lefschetz number of a monodromy $h$ given by the fibration $f: \C^{n+1}\setminus P\to \C^*$ we can  assume that the monodromy $h$ is  geometric, i.e. $h_f(z)=e^{2\pi i/d_f} z$. We start with the following lemma which  follows from \cite[Theorem 3.3 and Remark 3.4]{le}:

\begin{lemma}\label{lemat}
Let $f,g : \mathbb C^{n+1}\to  \mathbb C$ be two reduced holomorphic  functions. If the germs $(V(f), 0)$ and $(V(g),0)$ are ambient  homeomorphic, then the induced maps in homology of the monodromies of $f$ and $g$ at $0$ are conjugate. In particular, the monodromies of $f$ and $g$ at $0$ have the same Lefschetz number and $\chi(F_f ) = \chi(F_g ).$
\end{lemma}

Hence, for  $0 < k < d_f$, we have  $\Lambda (h_f^k) = 0$, where $\Lambda (h_f^k )$ denotes the Lefschetz number of $h_f^k.$ Similarly, for  $0 < k < d_g$, we have  $\Lambda (h_g^k) = 0.$ 
It follows from Lemma \ref{lemat} that $\Lambda(h_f^k) = \Lambda (h_g^k)$ for all natural $k.$ We have $\Lambda(h_f^{d_f})=\chi(F_f)$ and  $\Lambda(h_g^{d_g})=\chi(F_g)$.
Hence, if $\chi(F_f)=\chi(F_g)\not=0$,  we have $d_f=d_g$. Now consider a natural mapping $F_f\ni x\mapsto [x]\in \mathbb P^n\setminus X$, which is a covering of degree $d_f.$ Hence, $\chi(F_f)=d_f\chi ( \mathbb P^n\setminus X)$. Similarly, 
$\chi(F_g)=d_g\chi ( \mathbb P^n\setminus Y)$.
This means by Lemma \ref{lemat} that $\chi ( \mathbb P^n\setminus X)=\chi ( \mathbb P^n\setminus Y)$. But $\chi ( \mathbb P^n\setminus X)=n+1-\chi(X)$, $\chi ( \mathbb P^n\setminus Y)=n+1-\chi(Y)$, hence $\chi(X)=\chi(Y).$
If $\chi(F_f)=0$, then  by the same argument we have $\chi ( \mathbb P^n\setminus X)=\chi ( \mathbb P^n\setminus Y)=0$ and $\chi(X)=\chi(Y)=n+1.$
\end{proof}

\begin{remark}
{\rm If $S\subset \mathbb{C}^m$ is an irreducible homogeneous complex algebraic surface, it was proved in \cite{BobadillaFS:2018} that the torsion of $H^2(S\setminus \{0\})$ is equal to $\mathbb{Z}/{\rm deg} S \mathbb{Z}$. However, this does not extend to higher dimensions. Indeed, let $Q=\mathbb {CP}^{1}\times\mathbb {CP}^{1}\subset \mathbb {CP}^{3}$ be the quadric and $L$ the link of $C(Q)$. Then $L$ is simply connected (see the proof of \cite[Theorem 3.1] {FernandesJS:2023}). By \cite[Proposition 2.4]{bfsv}, $L$ is diffeomorphic to $\mathbb{S}^2\times\mathbb{S}^3$.  Hence, by the K\"unneth formula, $L$ has free homology (and cohomology). Therefore, the cohomology of $C(Q)\setminus\{0\}$ has no torsion.}
\end{remark}

\section{Answer to the Metric Arnold corank problem}

In this section, we prove our main result:
\begin{theorem}\label{thm:Lip_rank}
Let $f,g\colon (\C^n,0)\to (\C,0)$ be two reduced holomorphic functions that are singular at the origin. Assume that there is a bi-Lipschitz homeomorphism $\phi\colon (\C^n,0)\to (\C^n,0)$ that sends the zeros of $f$ onto the zeros of $g$. Then ${\rm crk}(f)={\rm crk}(g)$.
\end{theorem}
\begin{proof}
We divide our proof into two lemmas.
The first lemma states, in particular, that the theorem holds true if $m(V(f),0)=m(V(g),0)=2$.

\begin{lemma}\label{lemma:same_corank}
Let $f,g\colon (\C^n,0)\to (\C,0)$ be two reduced holomorphic functions. Assume that there is a bi-Lipschitz homeomorphism $\phi\colon (V(f),0)\to (V(g),0)$. If $m(V(f),0)=m(V(g),0)=2$, then ${\rm crk}(f)={\rm crk}(g)$.
\end{lemma}
\begin{proof}
Let $r=n-{\rm crk}(f)$ and $s=n-{\rm crk}(g)$. Then ${\rm crk}(f)={\rm crk}(g)$ if and only if $r=s$. Since $m(V(f),0)=m(V(g),0)=2$, we have that $r$ and $s$ are positive.
After a change of coordinates, if necessary, we may assume that $f$ and $g$ are written in the following way:
$$
f(z_1,...,z_r,z_{r+1},...,z_n)=z_1^{2}+...+z_r^{2}+h(z_{1},...,z_n)
$$
and 
$$
g(z_1,...,z_s,z_{s+1},...,z_n)=z_1^{2}+...+z_s^{2}+\tilde h(z_{1},...,z_n)
$$
where $h,\tilde h\colon (\C^{n},0)\to (\C,0)$ are holomorphic functions such that ${\rm ord}_0 h>2$ and ${\rm ord}_0 \tilde h>2$. 

 The case $r\le 2$ follows from \cite[Proposition 1.6]{FernandesS:2016} (see also \cite[Theorem 5.3]{FernandesS:2023}). So, by symmetry, we may assume that $r,s\geq 3$.

  Let $S$ be a tangent cone to $V(f)$ at $0$, and $T$ be a tangent cone to $V(g)$ at $0$. Hence, $S=\{ z\in \C^n: z_1^2+...+z_r^2=0\}$ and $T=\{ z\in \C^n: z_1^2+...+z_s^2=0\}$. In particular, 
 $\mathrm{Sing}(S)=\{0\}\times\C^{n-r} $ and 
 $\mathrm{Sing}(T)=\{0\}\times\C^{n-s}.$

 By \cite[Theorem 3.2]{Sampaio:2016}, there is a bi-Lipschitz homeomorphism $h\colon S\to T$.
 By \cite[Theorem 4.2]{Sampaio:2016}, $h(\mathrm{Sing}(S))=\mathrm{Sing}(T)$. Then, $\{0\}\times\C^{n-r} $ and  $\{0\}\times\C^{n-s}$ are bi-Lipschitz homeomorphic, and therefore $r=s$.
\end{proof}

So, we have to prove that if $m(V(f),0)=2$, then $m(V(g),0)=2$ as well. 

Assume that $m(V(f),0)=2$. Hence, in some coordinates $S=\{ z\in \C^n: z_1^2+...+z_r^2=0\},\  r\le n.$ Thus $S=S' \times \C^{n-r}$, where $S'=\{ z\in \C^r: z_1^2+...+z_r^2=0\}.$ In particular, 
 $\mathrm{Sing}(S)=\{0\}\times\C^{n-r}.$

 Let $S$ be a tangent cone to $V(f)$ at $0$ and $T$ be a tangent cone to $V(g)$ at $0$. By \cite{Sampaio:2016}, there is a bi-Lipschitz homeomorphism $\phi\colon \C^n\to\C^n$ such that $\phi(S)=T$.  
 
 The case $r\le 2$ follows from \cite[Proposition 1.6]{FernandesS:2016} (see also \cite[Theorem 5.3]{FernandesS:2023}). Hence, we can assume $r\ge 3$. 

 Therefore, it is enough to prove the following
 
 \begin{lemma}
 Let $S$ be an irreducible algebraic cone of degree two. Let $T$ be another algebraic cone. Assume that there is an ambient bi-Lipschitz mapping $\phi: (\mathbb{C}^n, {S})\to (\mathbb{C}^n, T).$ Then  $\deg T=2.$  Moreover, there exists a linear isomorphism
 $L: S\to T.$
 \end{lemma}
 
 \begin{proof}
We can assume that $S=\{ z\in \C^n: z_1^2+...+z_r^2=0\}, 3\le r\le n.$  Assume that the lemma is false. Let $d>2$ be the minimal degree of an algebraic cone $T$ such that there is a bi-Lipschitz mapping $\phi : S\to T.$ 
Since smoothness is 
 a bi-Lipschitz invariant and $T$ is bi-Lipschitz equivalent to $S$ we deduce that also $\mathrm{Sing}(T)$  in some system of coordinates is equal to $\{0\}\times \C^{n-r}$ (see  \cite{Sampaio:2016}  and \cite{prill}). Additionally, we can choose this system of coordinates so generic, that
 $T'\times\{0\}=T\cap (\C^r \times \{0\})$ is a cone with isolated singularity. 
 
 Now consider two points $z,w\in \mathrm{Sing}(S).$ Then $z, w \in  \{0\}  \times \C^{n-r}.$ There is a translation $L: \C^n\ni x\mapsto x+w-z\in \C^n$, which is of course bi-Lipschitz and which preserves $S$ and $\mathrm{Sing}(S).$ Hence $(\C^n, S,z)$ is bi-Lipschitz equivalent to
 $(\C^n,S, w).$  Since $(\C^n, S)$ is bi-Lipschitz equivalent to $(\C^n,T)$, the same is true for $T.$ 
 Indeed, we have  bi-Lipschitz mappings $\psi_v=\phi L_v\phi^{-1} : \C^n\to\C^n,$  where $v\in \C^n$ and $L_v:\C^n\ni x\mapsto x+v\in \C^n$, preserve $T$ and 
$ \mathrm{Sing}(T).$  
 
 In particular, if $z\in \mathrm{Sing}(T)$, then we have the bi-Lipschitz mapping $d \psi_{\phi^{-1}(z)} \circ\phi: (\C^n, T_0 S=S)\to (\C^n, T_z T)$  (here $T_zS$ and $T_zT$ denote the tangent cones at the point $z$ to $S$ and $T$, respectively, and $d \psi_{\phi^{-1}(z)}$ is the Sampaio derivation, see \cite{Sampaio:2016}). Let $T_z=T_zT$. By our assumption, we have either deg $T_z=\deg\ T$ or deg $T_z=2.$ If deg $T_z=\deg\ T$ on a dense subset of $\mathrm{Sing}(T)$, then  for every   $x\in \mathrm{Sing}(T)$, we have $m(T,x)=deg\ T,$  because $m(T,\cdot)$ is upper semicontinuous.
 
  Let $\overline{T}$ be a projective closure of $T.$ Since $m(\overline{T},\cdot)$ is upper semicontinuous, we have $m(\overline{T},z)=\deg T$ also for $z\in \overline{\mathrm{Sing}(T)}.$ Let $W$ be the part at infinity of  $\overline{\mathrm{Sing}(T)}.$ Take  $w\in W$ and  $a\in T'$ and consider the line $l=\overline{(w,a)}.$ By the Bezout Theorem, we have $l\subset T.$ This means that the projection $\pi : T\to T'$ along $\{0\} \times \C^{n-r} $ (i.e. the projection with  center $W$) has $\{a\}\times \C^{n-r}, \ a\in T',$ as fibers, i.e., $T=T'\times \C^{n-r}.$ This means that $S'^*$ and $T'^*$ have the same homotopy type. Thus, as in Proposition \ref{prop:1}, we have $\deg \ S=\deg \ T$,  a contradiction.
  
Hence, we can assume that $m(T,\cdot)=2$ on a dense subset of $\mathrm{Sing}(T)$. Take $z\in \mathrm{Sing}(T)= \{0\}\times\C^{n-r}.$ Since $T_z T$ is quadratic and bi-Lipschitz equivalent to $T$ and $\{0\}\times\C^{n-r} \subset T_z=T_z T,$ we have $\mathrm{Sing}(T_z)=\{0\}\times\C^{n-r}.$ In particular, for $x\in  \{0\}\times\C^{n-r}$, we have $m(T_z,x)=2.$ Since $m(\overline{T_z},\cdot)$ is upper semicontinuous, we have  $m(\overline{T_z},y)=2$ also for $y\in \overline{Sing(T_z)}.$ Let $T_z'\times\{0\}=T_z\cap (\C^{r}\times \{0\}).$ Thus the projection $\pi : T_z\to T'_z$ along $ \{0\}\times\C^{n-r}$ has $ \{a\}\times\C^{n-r}, a\in T'_z,$ as fibers, i.e., $T_z=T'_z\times \C^{n-r}.$ Denote coordinates in $ \{0\}\times\C^{n-r}$ by $x$ and coordinates in $ \C^r\times\{0\}$ by $t.$ Hence, $T_z$ has equation
$$\sum^r_{i,j=1} \alpha_{ij}(x)t_i t_j=0,$$ where $z=(0,x).$ Let $A=\det[\alpha_{ij}(x)].$  Note that $A$ is a non-zero polynomial on a dense subset of $\{0\}\times\C^{n-r}$, because all tangent cones of degree two have the same rank $r={\rm rank} \ S.$ Moreover,  $m(T,z)=2$ on $\C^{n-r}\setminus \{A=0\}$ (here we identify $ \{0\}\times\C^{n-r}$ with $\C^{n-r}$) and   all $\alpha_{ij}$ vanish on $A=0$, because the tangent cone, if it is quadratic, has to have the same rank as $S.$ Let $A=\prod^s_{j=1} a_j^{k_j}$ be a decomposition of $A$ into irreducible factors. 

Fix $i,j$ and consider the polynomial $\alpha_{i,j}(x):=\alpha.$ Thus $$\alpha:=b\prod^s_{j=1} a_j^{r_j},$$ where $\mathrm{GCD}(b,A)=1.$ We first show that $b$ is a constant. Let $f$ be a homogeneous reduced polynomial that describes $T.$ We know that $\frac{\partial^2 f}{\partial t_i\partial t_j}(x,0)=2\alpha$. Let us assume that $a_1$ has  minimal degree. We can choose a generic system of coordinates so that $\frac{\partial a_1}{\partial x_1}(x)\not=0.$   Hence, on $\{a_1=0\}$, we have 
$$\frac{\partial^{r_1} f}{\partial t_i\partial t_j\partial^{r_1} x_1}(0,x)=
C b(x)\prod^s_{i=2} a_i^{r_i+2} (\frac{\partial a_1}{\partial x_1}(x))^{r_1},
$$ 
where $C$ is a non-zero constant. Take a point $x\in \{a_1=0\}$ in such a way that $b(x)\not=0$,  $\frac{\partial a_1}{\partial x_1}(x)\not=0$ and $a_i(x)\not=0$ for $i>1.$ Then deg $T_{(0,x)}T$ has degree at most $d-(\deg \ b+\sum^s_{i=2}r_i \deg \ a _i)< d.$ This is a contradiction. Hence, $b$ is a constant. In a similar way, we can prove that  $s=1.$ Finally, we show that $a_1$ is a linear polynomial. In fact,  on $\{a_1=0\}$ we have:

$$\frac{\partial^{r_1+2} f}{\partial t_i\partial t_j\partial^{r_1} x_1}(0,x)=
C  ( \frac{\partial a_1}{\partial x_1}(x))^{r_1},$$
where $C$ is a non-zero constant. Take a point $x\in \{ a_1=0\}$ such that $\frac{\partial a_1}{\partial x_1}(x,0)\not=0$. Then $\deg T_{(x,0)}T$ has degree at most $2+r_1.$ If $a_1$ is not a linear polynomial, then $2+r_1<d$. This is a contradiction.  
We can assume that $a_1=x_1.$ Thus 
$$
\frac{\partial^2 f}{\partial t_i\partial t_j}(0,x)= x_1^{d-2}  c_{ij} ,
$$ 
where $c_{ij}$ are constants. 

Now consider $\frac{\partial^3 f}{\partial t_i\partial t_j\partial t_k}(x,0)=\beta.$  If $d>3$ we have again $\beta=0$ on $x_1=0$, because otherwise on $\{x_1=0\}\subset \mathrm{Sing}(T)$ we can find a tangent cone to $T$ of degree less than $d.$ As above, we see that
$\beta=Cx_1^{d-3}$ where $C$ is a constant. We can continue in this manner to obtain from the Taylor formula
$$f(x,t)=\sum^{d}_{k=2} \frac{1}{k!} \sum_{i_1,i_2,...,i_k} x_1^{d-k}  c_{i_1, i_2,..., i_k} t_{i_1} t_{i_2}... t_{i_k}.$$

On the other hand, we know that $T$ is bi-Lipschitz equivalent to $S$, in particular, by Proposition \ref{prop:euler} the bases of these two cones have the same Euler characteristic. This implies that also the bases of the cones 
$$\tilde{S}'=\{ (x,t)\in \C\times \C^r : t_1^2+...+t_r^2=0\}$$ and $$\tilde{T}'=\{ (x,t)\in \C\times \C^r : \sum^{d}_{k=2} \frac{1}{k!} \sum_{i_1,i_2,...,i_k} x^{d-k}  c_{i_1, i_2,..., i_k} t_{i_1} t_{i_2}... t_{i_k}=0\}$$ have the same Euler characteristic. 

Indeed,  let $B,C$ be  the bases of  $\tilde{S}',\tilde{T}'$ respectively. Consider a new cone 
$\C\times \tilde{S}'$ with base $B'.$ Then $B'$ in $\mathbb P^{r+1}$ is a cone over $B$ with  vertex $A=(1:0:...:0),$ i.e., it is a line bundle over $B$ plus one point. Hence, $\chi(B')=\chi(B)+1$. For the same reason, the base $C'$ of the cone $\C\times \tilde{T}'$ has Euler characteristic $\chi(C')=\chi(C)+1$. Thus, if the Euler characteristics of the new bases coincide, then the Euler characteristics of the old bases also coincide. The result is obtained by induction.

Observe that the bases  $B$ and $C$ have only one singular point $(1:0:0:....:0)$ each. Moreover, the Milnor number at this point is the same at $B$ and $C$, and it is equal to $1$, because both of these singularities are quadratic, i.e., they can be described as $\{z\in\C^r: z_1^2+...+z_r^2+{\it higher\ order\ terms\ in \ z} \}.$ 

Now we use the following theorem, which is an easy consequence of Parusi\'nski's result from  \cite{par}:

\begin{theorem}\label{par}
Let $X\subset \C\mathbb P^{n}$ be a degree $d$  complex projective hypersurface. Assume that $X$ has only isolated singular points $x_1,..., x_k$ and the Milnor number of  $X$ at $x_i$ is equal to $\mu_i.$ Then
the Euler characteristic of $X$ is given by the formula
$$ \chi (X) = (n+1) - \frac{1}{d} (1 + (-1)^{n}(d - 1)^{n+1})+ (-1)^n \sum^k_{i=1} \mu_i.$$
\end{theorem}

By Theorem \ref{par}, we have $r+1 - \frac{1}{d} (1 + (-1)^{r}(d - 1)^{r+1})\pm 1= r+1-\frac{1}{2} (1 + (-1)^{r}(2 - 1)^{r+1})\pm 1.$ Since the function  $f(x)=\frac{(x-1)^{r+1}+(-1)^{r}}{x}$ increases for $x\ge 2$, we get $d=2$, which is a contradiction.

Hence, deg $T=2.$ Moreover, the quadratic functions that describe $S$ and $T$ have the same rank equal to $n-\dim \mathrm{Sing}(S)=n-\dim \mathrm{Sing}(T).$ This means that they are linearly equivalent, hence there is a linear isomorphism $L:\C^n\to\C^n$ such that
$L(S)=T.$
\end{proof}
\end{proof}

As a direct consequence, we obtain that the multiplicity two is an ambient bi-Lipschitz invariant.
\begin{corollary}\label{cor:nowe}
Let  $(X,0),(Y,0)\subset (\C^n,0)$ be germs of analytic hypersurfaces that are ambient bi-Lipschitz equivalent. If $m(X,0)=2$, then also $m(Y,0)=2.$ 
\end{corollary}

\begin{remark}
  {\rm In Corollary \ref{cor:nowe}, it is important that $X,Y$ are hypersurfaces. Indeed, one can construct a quadratic three dimensional cone  $X\subset \C^7$ that is ambient bi-Lipschitz equivalent to a non-quadratic three dimensional cone  $Y\subset \C^7.$ This can be easily deduced from 
  \cite{bfsv} and \cite{bfj}.}
  \end{remark}
  

\section{Answer to the Arnold corank problem for surfaces}
In this Section, we prove that the Arnold corank problem has a positive answer for functions of three variables. 

Let us state some definitions and results that we need for our proof.

\begin{definition}
A holomorphic function $f\colon (\C^n,0)\to (\C,0)$
is said a {\bf generalized semi-Brieskorn-Pham} if there are a holomorphic diffeomorphism $\phi\colon (\C^n,0)\to (\C^n,0)$, a holomorphic function $f\colon (\C^n,0)\to (\C,0)$, and positive integer numbers $k, a_1,...,a_{k-1}$ and $a_k$ such that 
\begin{enumerate}
\item $a_1\leq a_2\leq ....\leq a_k$;
\item ${\rm ord}_0h>a_k$;
 \item $f\circ \phi(z_1,...,z_n)=z_1^{a_1}+...+z_k^{a_k}+h(z_1,...,z_n)$ for all $z=(z_1,...,z_n)$ close to $0$.
\end{enumerate}
In this case, if $a_1=1$, we define $m_1(f)=1$ and if $a_1>1$, we define $m_1(f):=\max\{i;a_i=a_1\}$.
\end{definition}

The following result was proved separately by A'Campo in \cite{Acampo:1973} and L\^e in \cite{Le:1973}.
\begin{theorem}[A'Campo-L\^e's Theorem]
Let $f,g\colon(\mathbb{C}^n,0)\to (\mathbb{C},0)$ be two complex analytic functions such that there is a homeomorphism $\varphi\colon(\mathbb{C}^n,V(f),0)\to (\mathbb{C}^n,V(g),0)$. If $m(V(f),0)=1$ then $m(V(g),0)=1$ as well.
\end{theorem}

The next result was proved by A'Campo in \cite{Acampo:1975} (see also \cite{Karras:1986}). \begin{theorem}\label{thm:acampo_cpmplement}
Let $f\colon(\mathbb{C}^n,0)\to (\mathbb{C},0)$ be a complex analytic function. Then
$$
m(V(f),0)\leq \inf \{k>0;\Lambda(h_f^k)\not=0\},
$$
where $h_f\colon F_f\to F_f$ is the monodromy of the Milnor fibration of $f$, and $\Lambda(h_f^k)$ denotes the Lefschetz number of $h_f^k$. Moreover, if $\chi(\C P^{n-1}\setminus PV(in(f)))\not=0$, then
$$
m(V(f),0)=\inf \{k>0;\Lambda(h_f^k)\not=0\},
$$
where $PV(in(f))=p(V(in(f))\setminus\{0\})$ and $p$ is the natural projection $p\colon \C^n\setminus\{0\}\to \C P^{n-1}$.
\end{theorem}

Now, we are ready to state and prove the main result of this Section.
\begin{theorem}\label{thm:top_rank}
Let $f,g\colon (\C^3,0)\to (\C,0)$ be two reduced holomorphic functions. Assume that there is a homeomorphism $\phi\colon (\C^3,0)\to (\C^3,0)$ such that $\phi(V(f))=V(g)$. Then ${\rm crk}(f)={\rm crk}(g)$.
\end{theorem}
\begin{proof}
Since ${\rm rk}(V(f),0)=3-{\rm crk}(f)$ and ${\rm rk}(V(g),0)=3-{\rm crk}(g)$, we obtain that ${\rm crk}(f)={\rm crk}(g)$ if and only if ${\rm rk}(V(f),0)={\rm rk}(V(g),0)$. So, we are going to show ${\rm rk}(V(f),0)={\rm rk}(V(g),0)$.

We first consider the case that ${\rm rk}(V(f),0)\not=0$ and ${\rm rk}(V(g),0)\not=0$.
\begin{claim}\label{claim:mult_two}
If ${\rm rk}(V(f),0)\not=0$ and ${\rm rk}(V(g),0)\not=0$, then ${\rm rk}(V(f),0)={\rm rk}(V(g),0)$.
\end{claim}
\begin{proof}[Proof of Claim \ref{claim:mult_two}]
Assume that ${\rm rk}(V(f),0)>0$ and ${\rm rk}(V(g),0)>0$.

Suppose that ${\rm rk}(V(f),0)=3$. Then, after a change of coordinates, if necessary, we may assume that $f(z_1,z_2,z_3)= z_1^2+z_2^2+z_3^2$. In particular, $\mu(f)=1$. Then, $\mu(g)=1$. Therefore, ${\rm rk}(V(g),0)=3$ as well.

Now, assume that ${\rm rk}(V(f),0)=2$. Then, by the Splitting Lemma \cite[Theorem 2.4]{GreuelP:2026}, after a change of coordinates, if necessary, we may assume that $f(z_1,z_2,z_3)= z_1^2+z_2^2+az_3^{k+1}$ for some integer  $k\geq 2$ and some constant $a\in \C$. Indeed, we may assume that $a\in \{0,1\}$.

If $a=1$, then the zeros of $f$, $V(f)$, is an $A_k$ singularity, and its link $L_f$ has finite fundamental group. Let $G=\pi(L_f)$. Then $(V(f),0)$ is isomorphic to $(\C^2/G,0)$ (see, for instance, \cite[Theorem 3.15(ii)]{Dimca:1992}).  Since $\pi_1(L_g)\cong G$, where $L_g$ is the link of the zeros of $g$, we have that $(V(g),0)$ is isomorphic to $(\C^2/G,0)$ (see, for instance, \cite[Corollary 3.16 and Theorem 3.15(ii)]{Dimca:1992}).
Therefore, $(V(g),0)$ is isomorphic to $(V(f),0)$. Then $m(V(g),0)=m(V(f),0)=2$ (see e.g. \cite{gau}, \cite{Sampaio:2019}, \cite{Sampaio:2020b} and \cite{Sampaio:2022}). In particular, $(V(g),0)$ is bi-Lipschitz homeomorphic to $(V(f),0)$. By Lemma \ref{lemma:same_corank}, ${\rm rk}(V(f),0)={\rm rk}(V(g),0)=2$.

If $a=0$, then $V(f)$ is the union of two planes $P_1$ and $P_2$. By A'Campo-L\^e's Theorem, $V(g)$ is the union of two smooth surfaces, which also implies ${\rm rk}(V(f),0)={\rm rk}(V(g),0)=2$.

By symmetry, if ${\rm rk}(V(f),0)=1$, then ${\rm rk}(V(g),0)=1$.

Therefore, in any case, ${\rm rk}(V(g),0)={\rm rk}(V(f),0)$, and thus ${\rm crk}(f)={\rm crk}(g)$. 
\end{proof}
To finish our proof, we need the following lemma, which we state in a much more general setting.
\begin{lemma}\label{lemma:semi_top_mult}
Let $f,g\colon (\C^n,0)\to (\C,0)$ be two holomorphic functions. Assume that $f\colon (\C^n,0)\to (\C,0)$ is a generalized semi-Brieskorn-Pham function and there is a homeomorphism $\phi\colon (\C^n,0)\to (\C^n,0)$ such that $\phi(V(f))=V(g)$. If $m(V(f),0)>2$ or $m_1(f)$ is an odd number, then $m(V(g),0)\leq m(V(f),0)$.
\end{lemma}
\begin{proof}
After a change of coordinates, if necessary, we may assume that $f$ is written in the following way:
$$
f(z_1,...,z_r,z_{r+1},...,z_n)=z_1^{a_1}+...+z_k^{a_k}+h(z_{1},...,z_n),
$$
where $h\colon (\C^{n},0)\to (\C,0)$ is a holomorphic function such that ${\rm ord}_0 h>a_k$. In this case, $in(f)(z)= z_1^{a}+...+z_r^{a}$, where $r=m_1(f)$. Let $q\colon \colon \C^r\to \C$ be the polynomial $q(z_1,...,z_r)=z_1^a+...+z_r^a$. Then $F_{in(f)}=F_q\times \C^{n-r}$. In particular, $\chi(F_{in(f)})=\chi(F_q)$. However, $\chi(F_q)=1+(-1)^{r-1}(a-1)^r=0$ if and only if $a=2$ and $r$ is an even number. Since we are assuming that $a>2$ or $r$ is an odd number, we have $\chi(F_{in(f)})\not=0$. Since the restriction to $F_{in(f)}$ of the natural projection $p\colon \C^n\setminus\{0\}\to \C P^{n-1}$ gives a finite covering map $p|\colon F_{in(f)}\to \C P^{n-1}\setminus PV(in(f))$, we obtain $\chi(\C P^{n-1}\setminus PV(in(f)))\not=0$. 

By Theorem \ref{thm:acampo_cpmplement}, we obtain that
$$
m(V(f),0)=\inf \{k>0;\Lambda(h_f^k)\not=0\}
$$
and
$$
m(V(g),0)\leq\inf \{k>0;\Lambda(h_g^k)\not=0\},
$$
where $h_f\colon F_f\to F_f$ (resp. $h_g\colon F_g\to F_g$) is the monodromy of the Milnor fibration of $f$ (resp. $g$), and $\Lambda(h)$ denotes the Lefschetz number of $h$.

Since $\Lambda(h_f^k)=\Lambda(h_g^k)$ for all $k$,
we have that $m(V(g),0)\leq m(V(f),0)$.
\end{proof}

Now, we are going to prove that ${\rm rk}(V(f),0)\not=0$ if and only if ${\rm rk}(V(g),0)\not=0$, which will complete the proof of the theorem. Indeed, this follows from a result stated in \cite{nav}, but as far as we know, a proof of that result was never published. For convenience, we present a proof here.

Assume, for instance, that ${\rm rk}(V(f),0)\not=0$. Then $m(V(f),0)=2$, and thus $f$ is a semi-Brieskorn-Pham function. In this case, ${\rm rk}(V(f),0)=m_1(f)$. If ${\rm rk}(V(f),0)$ is an odd number, then by Lemma \ref{lemma:semi_top_mult}, $m(V(g),0)\leq m(V(f),0)=2$. By A'Campo-L\^e's Theorem, $m(V(g),0)\geq 2$. Therefore, $m(V(g),0)=m(V(f),0)=2$.

Thus, we may assume that ${\rm rk}(V(f),0)=2$. By the Splitting Lemma again, we may assume that $f(x,y,z)=x^2+y^2+az^{k+1}$, where $k\geq 2$ and $a\in \{0,1\}$.

If $a=1$, then by proceeding as in the proof of Claim \ref{claim:mult_two}, we obtain that $(V(g),0)$ is isomorphic to $(V(f),0)$. In particular, we obtain $m(V(g),0)=m(V(f),0)=2$, and thus ${\rm rk}(V(g),0)\not=0$. 

If $a=0$, then $V(f)$ is the union of two planes $P_1$ and $P_2$. By A'Campo-L\^e's Theorem, $V(g)$ is the union of two smooth surfaces. Then $m(V(g),0)=2$, and thus ${\rm rk}(V(g),0)\not=0$. 
\end{proof}


\vspace{15mm} 


\begin{thebibliography}{99}
\bibitem{Acampo:1973} 
{A'Campo, N.}
{\it Le nombre de Lefschetz d'une monodromie}.
(French) Nederl. Akad. Wetensch. Proc. Ser. A 76 = Indag. Math., vol. 35 (1973), 113--118.

\bibitem{Acampo:1975} 
{A'Campo, N.}
{\it La fonction z\^eta d'une monodromie}.
Comment. Math. Helv., vol. 50 (1975), 233–248.

\bibitem{Arnold:2004}
{Arnold, V. I.}, 
{\it Arnold's problems}. 
Springer-Verlag, Berlin; PHASIS, Moscow, 2004.

\bibitem{aznar}  Aznar, N. V., 
{\em Sobre la invariancia topologica de la multiplicita}.
Pub. Sec. Mat. Univ. Autonoma Barcelona 20 (1980) 261–262. 


\bibitem{bfj}  {Birbrair, L.; Fernandes, A.; Jelonek, Z.}
{\em On the extension of bi-Lipschitz mappings}.
Selecta Mathematica , vol. 27:15, (2021).



\bibitem{bir}
{Birbrair, L.; Fernandes, A.; L\^e D. T. and Sampaio, J. E.}
{\em Lipschitz regular complex algebraic sets are smooth}.
Proc. Amer. Math. Soc., vol. 144 (2016), no. 3, 983-987.


\bibitem{bfsv}  Birbrair, L.; Fernandes, A.; Sampaio, J. E. and Verbitsky, M. 
{\it Multiplicity of singularities is not a bi-Lipschitz invariant}. Math. Ann., vol. 377 (2020), 115-121 


\bibitem{BobadillaFS:2018} 
Bobadilla, J.F. de; Fernandes, A. and Sampaio, J. E.,  
{\it Multiplicity and degree as bi-lipschitz invariants for complex sets}.
Journal of Topology, vol. 11 (2018), 958--966.

\bibitem{comte} 
Comte, G.: {\it Multiplicity of complex analytic sets and bi-Lipschitz maps.}  Real analytic and algebraic singularities (Nagoya/Sapporo/Hachioji, 1996) Pitman Res. NotesMath. Ser. 381 (1998), 182–188.

\bibitem{comteetall} 
Comte, G.; Milman, P. and Trotman, D. 
{\it On Zariski’s multiplicity problem}.
Proceedings of the American Mathematical Society 130 (2002), no. 7, 2045–2048.

\bibitem{Dimca:1992}
{Dimca, A.}
{\it Singularities and topology of hypersurfaces}.
New york: Springer-Verlag, 1992.

\bibitem{eph} Ephraim, R. 
{\it $C^1$ preservation of multiplicity}. 
Duke Math., vol. 43 (1976), 797–803.

\bibitem{F}
{Fernandes, A.}
{\it Topological equivalence of complex curves and bi-Lipschitz maps}.
{Michigan Math. J.}, vol. 51 (2003), 593--606.


\bibitem{FernandesJS:2023} Fernandes, A.; Jelonek, Z. and Sampaio, J. E. 
{\it Bi-Lipschitz equivalent cones with different degrees.} 
Preprint (2023), arXiv:2309.07078 [math.AG]

\bibitem{FernandesS:2016}
{Fernandes, A. and Sampaio, J. E.}.
{\it Multiplicity of analytic hypersurface singularities under bi-Lipschitz homeomorphisms}.
Journal of Topology, vol. 9 (2016), 927--933.

\bibitem{FernandesS:2023} Fernandes, A. and Sampaio, J.E. 
{\it Bi-Lipschitz Invariance of the Multiplicity.} In: Cisneros-Molina, J.L., D\~ung Tr\'ang, L., Seade, J. (eds) Handbook of Geometry and Topology of Singularities IV. Springer, 2023.

\bibitem{gau}  
Gau, Y. and Lipman, J.  
{\it Differential invariance of multiplicity on analytic varieties}. 
Invent. Math., vol. 73 (1983), 165–188.

\bibitem{GreuelP:2026}  
Greuel, G.-M. and Pfister, G.
{\it The splitting lemma in any characteristic}.
Journal of Algebra, vol. 689 (2026), 610-628.

\bibitem{Karras:1986}
Karras, U.
{\it Equimultiplicity of deformations of constant Milnor number}.
Proceedings of the Conference on Algebraic Geometry (Berlin 1985), Teubner-Texte Math. 92, Teubner, Leipzig, 1986, 186–209.

\bibitem{kol} Kollar, J., 
{\it Links of complex analytic singularities}.  In: Surveys in Differential Geometry XVIII.
International Press of Boston, (2013), 157--192.


\bibitem{le} L\^e D. T. 
{\it Topologie des singularit\'es des hypersurfaces complexes}.  
Singularit\'es \'a Carg\'ese, Asterisque, vol. 7 and 8 (1973), 171–182.

\bibitem{Le:1973} 
{L\^e D. T.}
{\it Calcul du nombre de cycles \'evanouissants d'une hypersurface complexe}. 
(French) Ann. Inst. Fourier (Grenoble), vol. 23 (1973), no. 4, 261--270.
 
\bibitem{max} Maxim, L.G., 
{\it On the topology of complex projective hypersurfaces}. 
Res. Math. Sci., vol. 11 (2024), article no. 14, 1--22.


\bibitem{mil} Milnor, J, ,
{\rm Singular points of complex  hypersurfaces}. 
Annals of Mathematical Studies, Princeton University Press, Princeton, New Jersey, 1968.

\bibitem{nav}
Navarro Aznar, V.
{\it Sobre la invariancia  topologica de la multiplicitat}, Pub. Sec. Mat. Univ.
Autonoma Barcelona, vol. 20 (1980) 261–262.


\bibitem{N-P}
{Neumann, W. and Pichon, A.}
{\it Lipschitz geometry of complex curves}.
 Journal of Singularities, vol. 10 (2014), 225--234.
 




\bibitem{orlik} Orlik, P. and  Wagreich, P.
{\it Seifert n-manifolds}.
Invent. Math., vol. 28 (1975), 137--159.


 
\bibitem{par} {Parusi\'nski, A.} 
{\it A generalization of the Milnor number},  Math. Ann., vol. 281, (1988), 247-254. 



\bibitem{prill} Prill, D.
{\it Cones in Complex Affine Space are Topologically Singular}. 
Proc. Amer. Math. Soc., vol. 18, No. 1 (1967), 
178--182.


\bibitem{risler}  Risler,J.-J., Trotman D., {\it Bilipschitz invariance of the multiplicity}, Bull. London. Math.
Soc. 29 (1997), 200–204. 

\bibitem{Sampaio:2016}
{Sampaio, J. E.}
{\it Bi-Lipschitz homeomorphic subanalytic sets have bi-Lipschitz homeomorphic tangent cones}.
Selecta Math. (N.S.), vol. 22 (2016), no. 2, 553--559. 

\bibitem{Sampaio:2019}
{Sampaio, J. E.}
{\it A proof of the differentiable invariance of the multiplicity using spherical blowing-up}. 
Revista de la Real Academia de Ciencias Exactas, F\'isicas y Naturales. Serie A. Matem\'aticas, vol. 113 (2019), no. 4, 3913-3920.

\bibitem{Sampaio:2020b}
{Sampaio, J. E.}
{\it Multiplicity, regularity and blow-spherical equivalence of complex analytic set}. 
The Asian Journal of Mathematics, vol. 24 (2020), no. 5, 803--820.

\bibitem{Sampaio:2022}
{Sampaio, J. E.}
{\it Differential invariance of the multiplicity of real and complex analytic sets.} 
Publicacions Matem\`atiques, vol. 66 (2022), 355–368.


\bibitem{zar}
{Zariski, O.}
{\it Some open questions in the theory of singularities}. 
Bull. Amer. Math. Soc., vol. 77 (1971), no. 4, 481-491.

\bibitem{zar1}
{Zariski, O.}
{\it Studies in equisingularity. I . Equivalent singularities of plane algebroid curves}.  
Amer. J. Math., vol. 87 (1965), 507-536. 

\end{thebibliography}
\end{document}